\documentclass[11pt,reqno]{amsart}

\usepackage{setspace}

\usepackage{a4,latexsym,amssymb,amsfonts}
\usepackage{mathrsfs}
\usepackage{enumerate}
\usepackage{amsthm}
\usepackage{amsmath}
\usepackage{color}
\usepackage{pstricks-add}

\newtheorem{thm}{Theorem}[section]
\newtheorem{cor}[thm]{Corollary}
\newtheorem{exm}{Example}

\newtheorem{lem}[thm]{Lemma}
\newtheorem{prop}[thm]{Proposition}
\theoremstyle{definition}
\newtheorem{defn}[thm]{Definition}
\theoremstyle{remark}
\newtheorem{rem}[thm]{Remark}
\numberwithin{equation}{section}

\begin{document}

\title[]{Sufficient and Necessary Conditions for the Solvability of the State Feedback Regulation Problem}
\author{S. Boulite,   H. Bouslous, L. Maniar and R. Saij }

\address{{\it H. Bouslous, L. Maniar  and R. Saij}\\
University Cadi Ayyad, Faculty of Sciences Semlalia\\ Department of Mathematics
\\B.P. 2390, 40000 Marrakesh, Morocco.}

\email{bouslous@ucam.ac.ma,
maniar@ucam.ac.ma, saijrachid@gmail.com}
\address{ {\it S. Boulite}\\
 University Hassan II Casablanca, Faculty of
Sciences Ain Chock\\ Department of Mathematics \& Informatics \\B.P. 5366 Maarif 20100 Casablanca,  Morocco.}
\email{boulite@gmail.com}

\thanks{}
\maketitle

{\bf Abstract.} {\small
In this paper, we discuss the state feedback output regulation problem (SFRP) for infinite-dimensional linear control systems with infinite-dimensional exosystems. Under the polynomial stabilizability assumption, sufficient and necessary conditions are given for the solvability of the SFRP. The solvability of this problem is characterized in terms of the solvability of a pair of linear regulator equations. An application of the solvability of the SFRP for polynomial stable SISO system is given.
}


\section{INTRODUCTION}
In this paper, we are interested on the state feedback output regulation problem of the system described by the following equations
\begin{equation}\label{plant}
\left\{\begin{array}{cclccccc}
\dot{z}(t) & = & Az(t)+Bu(t)+\mathcal{U}_{d}(t), &  &t\geq0 && \\
y(t) & = & Cz(t), &  &t\geq0&& \\
z(0) & = & z_{0}.
\end{array}\right.
\end{equation}
Here $A$ generates a  $C_{0}$-semigroup $T_{A}(t)$, $t\geq0$, on a complex Banach space $Z$.
 The state of the plant (\ref{plant}) is denoted by $z(t)\in Z$. The continuous input $u: \mathbb{R}_{+}\longrightarrow U$ and the continuous output
 $y: \mathbb{R}_{+}\longrightarrow Y$ take values in a complex Banach spaces $U$ and $Y$ respectively. The control operator $B\in \mathcal{L}(U,Z)$ and the observation operator $C\in \mathcal{L}(Z,Y)$. The bounded uniformly continuous function
 $\mathcal{U}_{d}: \mathbb{R}\longrightarrow Z$ represents a disturbance.

In addition, we assume that there exists an infinite-dimensional linear system, referred to as the exogenous system (or exosystem), that generates a bounded uniformly continuous reference signals $y_{r}$ and  disturbance signals $\mathcal{U}_{d}$
\begin{equation}\label{exo}
\left\{\begin{array}{cccccccc}
\dot{w}(t) & = & Sw(t),& & t\in \mathbb{R}&&  \\
y_{r}(t) & = & Qw(t) ,& &t\in \mathbb{R}&&\\
\mathcal{U}_{d}(t) & = & Pw(t),  &  &t\in \mathbb{R}&&\\
w(0)&=& w_{0}.
\end{array}\right.
\end{equation}
Here $S$ generates an isometric $C_{0}$-group $T_{S}(t),t\in \mathbb{R},$ on a Banach space $W$, $P\in \mathcal{L}(W,Z)$ and $Q\in \mathcal{L}(W,Y)$.
We denote the error between the measured and reference outputs by $$e(t):=y(t)-y_{r}(t)=Cz(t)-Qw(t).$$
In general, the output regulation problem involves the construction of a control law which stabilizes  the plant (\ref{plant})  and drives the measured output  to   achieve asymptoticaly  the  reference signal  $y_{r}$ inspite of the disturbances $\mathcal{U}_{d}$. i.e.,  $e(t)\longrightarrow0, \quad t\longrightarrow\infty$.

For finite-dimensional linear systems, the SFRP was studied by Davison, Francis, Wonham and others (see e.g. \cite{{Davison1},{Davison2},{Davison3},{Francis}} and the references therein). In \cite{Francis}, Francis  presented a complete characterization for  SFRP in terms of solvability of the so called regulator equations\begin{equation}\label{reg eq}
\left\{\begin{array}{ccccccc}
A\Pi+B\Gamma+P & = & \Pi S&in\quad D(S)&&&  \\
C\Pi & = & Q &in\quad W.
\end{array}\right.
\end{equation}
In \cite{Davison3}, Davison  used  similar method for the construction of the control law which regulates the measured output.
Many authors  also constructed such control law for infinite dimensional linear plants with finite dimensional exosystems, e.g. Pohjolainen \cite{Pohjolainen} and Byrnes et al. \cite{Byrnes}. Under the exponential stabilizability assumption of the system (\ref{plant}), they proved a complete characterization for existence (and construction) of a regulating control law in terms of solvability of the regulator equations (\ref{reg eq}).  Subsequently, Immomen and Pohjolainen have considered infinite dimensional exosystems generating periodic reference signals in  \cite{Immonen0}.  In the same year,  
in  \cite{Immonen01}, Immonen and Pohjolainen generalized the above results  to the  case of   strongly stabilizable plants (\ref{plant}) and bounded uniformly continuous exosystems (\ref{exo}). More exactely, 
they proved that if the pair $(A,B)$  is strongly stabilizable by a feedback operator  $K$ and if the regulator equations (\ref{reg eq}) have a solution ($\Pi,\Gamma$), then the SFRP is  solved  by the control law $u(t)=Kz(t)+(\Gamma-K\Pi)w(t)$.
On the other hand, the converse problem, or the necessary condition,  is also  studied in  the case of finite dimensional exosystems and exponentially stabilizable   plants (\ref{plant}) in  \cite{Byrnes}. But, In the case of strong stabilizability,  Immonen and Pohjlainen \cite{Immonen01} (see \cite[Chapter 3]{Immonen})   needed   additional assumptions. First, they introduced the concept of regular operators. They  gave a characterization of regular operator via the operator equation
\begin{equation}\label{op eq}
A\Pi+\Delta=\Pi S
\end{equation}
and established that the operator $\Delta\in\mathcal{L}(W,Z)$ is regular for $T_{A}(t)$ if and only if the operator equation (\ref{op eq}) has a solution $\Pi\in\mathcal{L}(W,Z)$. This allowed   to solve the first regulator equation in (\ref{reg eq}). By imposing certain auxiliary conditions for the reference signals, they established that the second regulator equation of (\ref{reg eq}) were verified.

Notice that the operator equation (\ref{op eq}), refered as Sylvester equation,  was studied by many authors. In particular, in \cite{Phong}, Phong established that if $T_{A}(t)$ is exponentially stable then the operator equation (\ref{op eq}) has a unique bounded solution $\Pi:\;W\longrightarrow Z$ given by
\begin{equation}\label{expression de pi}
\Pi w=\int_{0}^{\infty}T_{A}(t)\Delta T_{S}(-t)w\,dt
\end{equation}
for all $w\in W$. Further, if $A$ generates a strongly stable $C_{0}$-semigroup, then the operator equation (\ref{op eq}) does not necessarily  have a solution. Hence, one may ask: what about the solution of (\ref{op eq}) if $T_{A}(t)$ is polynomially stable? In this work, we show that, for every $\Delta\in\mathcal{L}(W,Z)$, if the operator equation has a bounded solution $\Pi:\;W\longrightarrow Z$ then necessarily the operator $\Pi$ is given by the same formula as (\ref{expression de pi}) for all $w\in D(S)$. Then, we will call  every operator  $\Delta$ such that the integral in (\ref{expression de pi}) converges,  for all $w\in D(S)$, a conform operator for $T_{A}(t)$. In this case, we show that  the operator equation (\ref{op eq}) is solved by (\ref{expression de pi}). 

This paper is concerned with  output regulation problem  in the case of polynomially stabilizable  plants. Our SFRP can be formulated as follows:
let $y_{r}$ be a given signal reference. The task is to find a feedback control law of the form $$u(t)=Kz(t)+Lw(t),$$ 
for some $K\in \mathcal{L}(Z,U)$ and  $L\in \mathcal{L}(W,U)$, such that 
\begin{itemize}
\item[$\bullet$] $A+BK$ generates a bounded polynomially stable $C_{0}$-semigroup $T_{A+BK}(t)$ on $Z$.
\item[$\bullet$] For the extended closed loop system
\end{itemize}
\begin{equation}
\left\{\begin{array}{lll}
\dot{z}(t) & = (A+BK)z(t)+(BL+P)w(t),& t\geq0 \\
\dot{w}(t)& = Sw(t),& t\geq0
\end{array}\right.
\end{equation}
the tracking error $e(t)=Cz(t)-Qw(t)\longrightarrow0$ as $t\longrightarrow\infty$
for any initial conditions $z_{0}\in Z$ and $w_{0}\in W$, and $\|e(t)\|_{Y}\leq m\,t^{-\frac{1}{\alpha}},\quad t>0$ for $z_{0}\in D(A)$ and $w_{0}\in D(S)$.

Here, we are studying the problem of output regulation problem under the assumption of polynomial stabilization. In our knowledge, this last property is not studied in the litterature, and it is now under study and recent results will appear in a forthcoming paper. 
Note also that we have assumed that the closed loop semigroup is  aslo bounded for some technical problems. The general case is under study.

In this paper, we give first a 
sufficient  condition for the solvability of  SFRP under the polynomial stabilizability assumption of the plant, similar to the one of  Immonen and Pohjolainen \cite{Immonen01}.  To give  necessary conditions, we  use the conform operator notion introduced above and, under some assumptions, we give a characterization for the solvability
 of  SFRP and the regulator equations. In order to illustrate the obtained results, under some assumptions, we will solve explicitly  SFRP and the regulator equations for a diagonalizable $SISO$ system. Finally, an example of periodic tracking for a controlled wave equation is given.

\section{Preliminaries on polynomially stable $C_{0}$-semigroups}
In this section we fix our notations and review some results on polynomially stable $C_{0}$-semigroups. By $D(A)$, $\sigma(A)$, $\rho(A)$,
 we denote the domain, the spectrum, the resolvent set of a linear operator $A$, respectively, and we set $R(\lambda,A)=(\lambda I-A)^{-1}$
 for $\lambda\in \rho(A)$. The open left half-plane of $\mathbb{C}$ is denoted by $\mathbb{C}^{-}$.
 Throughout this section, $A$ is the generator of a $C_{0}$-semigroup $T(t)$ on a Banach space $Z$.
 Fix a real number $\mu$ such that $\|T(t)\|\leq M\,e^{(\mu-\varepsilon)t}$ for some constants $M,\varepsilon>0$ and all $t\geq0$.
 The fractional powers of $A_{\mu}:=\mu I-A$ are defined by
 $$A_{\mu}^{-\alpha}=\frac{1}{2\pi i}\int_{\Gamma}(\mu-\lambda)^{-\alpha}R(\lambda,
A)\,d\lambda,$$
for any  $\alpha>0$ and $\Gamma$ is any piecewise smooth path in the set \, $\{\lambda\in \mathbb{C}: Re\lambda>\mu-\varepsilon, \lambda\notin[\mu,\infty)\}$
 running from $\infty e^{-i\phi}$ to $\infty e^{i\phi}$ for some $0<\phi<\pi/2$. We further set $A^{0}_{\mu}=I$. The operator $A_{\mu}^{-\alpha}$
 is injective and bounded, hence it has a closed inverse denoted by $A_{\mu}^{\alpha}$. The domain $Z_{\alpha}:=D(A_{\mu}^{\alpha})$ is independent
of the choice of $\mu$. The domains $Z_{\alpha}$ endowed with the norm $\|z\|_{\alpha}=\|A_{\mu}^{\alpha}z\|_{Z}$, $\alpha\geq0$, $Z_{0}=Z$,  are Banach spaces. Observe that $Z_{\gamma}$ is continuously and densely embedded in $Z_{\alpha}$ for $\gamma\geq\alpha\geq0$ and that $\|z\|_{n}$ is equivalent
 to the usual graph norm of $A^{n}$ for $n\in \mathbb{N}$ which is denoted, in general, by $\|\cdot\|_{A}$. Moreover, the fractional powers commute with $T(t)$ and $A$.
 \begin{defn}\cite{Batkai}
 A $C_{0}$-semigroup $T(t)$, $t\geq0$ is called polynomially stable if there are constants $\alpha, \beta>0$ such that
 \begin{equation}\label{def.poly}
 \|T(t)A_{\mu}^{-\alpha}\|\leq N\,t^{-\beta}
 \end{equation}
 for some constant $N>0$ and all $t>0$.
 \end{defn}
 We denote that the inequality (\ref{def.poly}) is equivalent to
 \begin{equation}
 \|T(t)z\|\leq N\,t^{-\beta}\|z\|_{\alpha}
 \end{equation}
 for all $z\in Z_{\alpha}=D(A_{\mu}^{\alpha})$.
 Note that the above definition is independent of $\mu$ and that the estimate (\ref{def.poly}) with $\alpha=0$ and $\beta>0$ already implies that $T(t)$, $t\geq0$ is exponentially stable, i.e.
 \begin{equation*}\|T(t)\|\leq M\,e^{-a t}\end{equation*} for all $t\geq0$ and some constants $M,a>0$.
  It was shown in \cite{Batkai} that a polynomially stable $C_{0}$-semigroup satisfies
 \begin{equation}\label{polywithgamma}
 \|T(t)A_{\mu}^{-\alpha \gamma}\|\leq N(\gamma)\,t^{-\beta \gamma}
 \end{equation}
 for each $\gamma\geq1$. Moreover, inequality (\ref{polywithgamma}) holds for all $\gamma>0$ if $T(t),t\geq0$ is polynomially stable and bounded, i.e.
 \begin{equation}
 \|T(t)\|\leq M
 \end{equation}
 for some constant $M\geq1$ and all $t\geq0$. In this case, we have
 \begin{equation}\label{polybound}
 \|T(t)A_{\mu}^{-\alpha}\|\leq \frac{N}{t}
 \end{equation}
 for all $t>0$ (with a different $\alpha$, in general). Due to \cite[Proposition 3.3]{Batkai}, we have $\sigma(A)\subset\mathbb{C}^{-}$ and therefore we may normalize (\ref{polybound}) to the estimate \begin{equation}\label{polywithZalpha}
 \|T(t)(-A)^{-\alpha}\|\leq \frac{N}{t}
 \end{equation}
 for all $t>0$ and some $\alpha>0$, or equivalently
 \begin{equation}\label{polywithD(A)}
 \|T(t)A^{-1}\|\leq \frac{N}{t^{1/\alpha}}
 \end{equation}
for all $t>0$ and some $\alpha>0$. Finally, by density argument, we remark that a bounded polynomially stable $C_{0}$-semigroup is always strongly stable.
\section{sufficient conditions for the solvability of the (SFRB)}\label{section suffcond}
  In this section, under the assumption that $A+BK$ generates a bounded polynomially stable $C_{0}$-semigroup $T_{A+BK}(t)$, we shall give sufficient conditions for the solvability of the SFRP. Before given our main result in this section, we recall first, from \cite{Immonen01,Immonen}, the following lemma, and for the sake of completeness, we give a short proof.
  \begin{lem}
  Let $Z$ and $W$ be Banach spaces, let $A$ generates a $C_{0}$-semigroup $T_{A}(t)$ on $Z$ and $S$ generates a $C_{0}$-semigroup $T_{S}(t)$ on $W$
  and let $\Delta\in\mathcal{L}(W,Z)$. If there exists $\Pi\in\mathcal{L}(W,Z)$ such that $\Pi(D(S))\subset D(A)$ and $\Pi$ satisfies the Sylvester
    type operator equation (\ref{op eq}), namely
    \begin{equation*}
    \Pi S=A\Pi+\Delta
    \end{equation*}
    then
    \begin{equation}\label{lemmaprelim}
    \int_{0}^{t}T_{A}(t-\sigma)\Delta T_{S}(\sigma)w_{0}d\sigma=\Pi T_{S}(t)w_{0}-T_{A}(t)\Pi w_{0}
    \end{equation}
    for all $t\geq0$, $w_{0}\in W$.
  \end{lem}
   \begin{proof}
   Let $\omega_{0}\in D(S)$, by (\ref{op eq}), we have
   $$\begin{array}{cll}
         \int_{0}^{t}T_{A}(t-\sigma)\Delta T_{S}(\sigma)w_{0}d\sigma& = & \int_{0}^{t}T_{A}(t-\sigma)(\Pi S-A\Pi)T_{S}(\sigma)w_{0}d\sigma \\
         &&\\
     & = &  \int_{0}^{t}\frac{d}{d\sigma}T_{A}(t-\sigma)\Pi T_{S}(\sigma)w_{0}d\sigma \\
     &&\\
     & = & \Pi T_{S}(t)w_{0}-T_{A}(t)\Pi w_{0}
  \end{array} $$
  for all $t\geq0$.
  From another hand, for each $t\geq0$, it is clear that the operators $R_{1}(t)$ and $R_{2}(t)$ defined by
 \begin{equation*}
  R_{1}(t)w:=\int_{0}^{t}T_{A}(t-\sigma)\Delta T_{S}(\sigma)w d\sigma
  \end{equation*}
  \begin{equation*}
  R_{2}(t)w:=\Pi T_{S}(t)w-T_{A}(t)\Pi w
 \end{equation*}
  are in $\mathcal{L}(W,Z)$. Since $D(S)$ is dense in $W$ and $R_{1}(t)w=R_{2}(t)w$ for all $w\in D(S)$ and all $t\geq0$, we can extend the equality
  $R_{1}(t)w=R_{2}(t)w$ for each $t\geq0$ to hold for every $w\in W$.
  \end{proof}
  Our main result in this section is the following.
  \begin{thm}
  Assume that $A+BK$ generates a bounded polynomially stable $C_{0}$-semigroup $T_{A+BK}(t)$. If the regulator equations (\ref{reg eq}), namely
  \begin{equation*}
  \left\{\begin{array}{ccccccc}
    A\Pi+B\Gamma+P & = & \Pi S&in\quad D(S)&&&  \\
    C\Pi & = & Q &in\quad W&&&
  \end{array}\right.
  \end{equation*}
  have a solution, then the control law $u(t)=Kz(t)+(\Gamma-K\Pi)w(t)$ solves  SFRP.
  \end{thm}
  \begin{proof}
  Since $A+BK$ generates the bounded polynomially stable $C_{0}$-semigroup $T_{A+BK}(t)$, we only need to verify the second condition of SFRP.
  Let $L=\Gamma-K\Pi\in\mathcal{L}(W,U)$. Then
  \begin{equation}\label{op eq tablized}
  \Pi S=(A+BK)\Pi+BL+P\quad in\,D(S).
  \end{equation}
  Hence, by Lemma \ref{lemmaprelim}, we have
  \begin{equation}\label{lemmaprelimwith K}
   \int_{0}^{t}T_{A+BK}(t-\sigma)(BL+P)T_{S}(\sigma)w_{0}d\sigma=\Pi T_{S}(t)w_{0}-T_{A+BK}(t)\Pi w_{0}
  \end{equation}
  for every $w_{0}\in\,W$ and every $t\geq0$.
  Now consider the feedback law $$u(t)=Kz(t)+Lw(t)$$ and let $z_{0}\in Z$ and $w_{0}\in W$ be arbitrary.
  Due to \cite[Theorem 3.6]{Immonen}, the explicit expression for the tracking error $e(t)$ is as follows
\begin{equation}\label{errorexpreswith(CPi-Q)}
e(t)=CT_{A+BK}(t)(z_{0}-\Pi w_{0})+(C\Pi-Q) T_{S}(t)w_{0}.
\end{equation}
By the second equation of \eqref{reg eq}, we have $C\Pi  =  Q \;in\; W$, then
\begin{equation}\label{errorexpres}
e(t)=CT_{A+BK}(t)(z_{0}-\Pi w_{0}),\qquad t\geq0.
\end{equation}
Since $T_{A+BK}(t)$ is polynomially stable and bounded then $T_{A+BK}(t)$ is strongly stable. Hence, by the boundedness of the operator $C$, we obtain that the tracking error $e(t)$ converges to $0$ as $t\longrightarrow \infty$ for all $z_{0}\in Z,\,w_{0}\in W.$
\end{proof}
 As an immediate consequence of the formula \eqref{errorexpres}, we can give the rate of decay of $\|e(t)\|_{Y}$ for  more regular  initial conditions.
 \begin{cor}
  Assume that $z_{0}\in D(A)$ and $w_{0}\in D(S)$. Then there is a positive constant $m$ depending only on $\|C\|_{\mathcal{L}(Z,Y)},\|z_{0}\|_{A}$ and $\|w_{0}\|_{S}$ such that
 \begin{equation}
 \|e(t)\|_{Y}\leq m\,t^{-\frac{1}{\alpha}},\quad\quad \forall t>0.
 \end{equation}\end{cor}
 \begin{proof}
 From the formula (\ref{errorexpres}), we have
 \begin{equation}\label{errorestimate}
 \|e(t)\|_{Y}\leq\|C\|_{\mathcal{L}(Z,Y)}\|T_{A+BK}(t)(z_{0}-\Pi w_{0})\|.
 \end{equation}
 Since $T_{A+BK}(t)$ is polynomially stable and $z_{0}-\Pi w_{0}\in D(A+BK)$ then, by inequality (\ref{polywithD(A)})
 \begin{equation}\label{semigrinequality}
 \|T_{A+BK}(t)(z_{0}-\Pi w_{0})\|\leq N t^{-\frac{1}{\alpha}}\|z_{0}-\Pi w_{0}\|_{A+BK}
 \end{equation}
 for all $t>0$. Remark that $\Pi\in\mathcal{L}(D(S),D(A+BK))$. In fact, if we put $\Delta:=BL+P$, then by (\ref{op eq tablized}), for every $w\in D(S)$, we have
 $$\begin{array}{ccl}
   \|\Pi w\|_{A+BK} & = & \|\Pi w\|+\|(A+BK)\Pi w\|\\ 
    & = & \|\Pi w\|+\|\Pi Sw-\Delta w\|\\ 
    & \leq & \|\Pi\|\|w\|+\|\Pi\|\|Sw\|+\|\Delta\|\|w\|\\ 
     & \leq & (\|\Pi\|+\|\Delta\|)\|w\|_{S}=C^{'}\|w\|_{S}\label{boundedPiD(A)D(S)}.
 \end{array}$$
Hence, from the inequalities (\ref{errorestimate}) and (\ref{semigrinequality}), we deduce that
  \begin{equation}
 \|e(t)\|_{Y}\leq\|C\|_{\mathcal{L}(Z,Y)}Nt^{-\frac{1}{\alpha}}(\|z_{0}\|_{A}+C^{'}\|w_{0}\|_{S})
 \end{equation} for all $t>0$.
 Thus
  \begin{equation*}
 \|e(t)\|_{Y}\leq m\,t^{-\frac{1}{\alpha}}
 \end{equation*}
 with
 \begin{equation*}
m=\|C\|_{\mathcal{L}(Z,Y)}N (\|z_{0}\|_{A}+C^{'}\|w_{0}\|_{S}).
 \end{equation*}
 \end{proof}
 \section{Necessary conditions for the solvability of SFRP}
 In this section, we shall discuss necessity of solvability of the regulator equations (\ref{reg eq}) for the solvability of the SFRP.
 To this end, we introduce a  concept of conform operators $\Delta\in \mathcal{L}(W,Z)$. This allows us to give a complete characterization of the solvability of the Sylvester type operator equation (\ref{op eq}). The solvability of the second regulator equation of (\ref{reg eq}) is subsequently obtained by imposing certain auxiliary conditions for the reference signals or the speed of output regulation. We emphasize that such conditions are the same imposed by Immonen in \cite{Immonen}.
 \begin{defn}
 Let A generate a polynomially stable $C_{0}$-semigroup $T_{A}(t)$. An operator $\Delta\in \mathcal{L}(W,Z)$ is said to be conform for the semigroup $T_{A}(t)$ if the operator $$\Pi\omega:=\int_{0}^{\infty}T_{A}(t)\Delta T_{S}(-t)w dt, \quad w\in D(S),
$$ 
define a linear bounded operator from $D(S)$ to $Z$, where $D(S)$ is endowed with the induced norm of $W$.
\end{defn}
 In the following lemma, we give a characterization of a conform operator via  Sylvester type operator equation (\ref{op eq}).
 \begin{lem}\label{characonform}
 Let $T_{A}(t)$ be a polynomially stable $C_{0}$-semigroup and let $\Delta\in \mathcal{L}(W,Z)$. The operator $\Delta$ is conform for $T_{A}(t)$ if and only if the operator equation $\Pi S=A\Pi+\Delta$ has a solution $\Pi\in \mathcal{L}(W,Z)$.
 \end{lem}
 \begin{proof}
 Let $\Pi\in \mathcal{L}(W,Z)$ such that $\Pi(D(S))\subset D(A)$ and
 \begin{equation*}
 \Pi S=A\Pi+\Delta \;in \;D(S).
 \end{equation*}
 From Lemma \ref{lemmaprelim}, we have
 \begin{equation}\label{expressPiwith sigma}
 \Pi w=T_{A}(\sigma)\Pi T_{S}(-\sigma)w+\int_{0}^{\sigma}T_{A}(t)\Delta T_{S}(-t)wdt
 \end{equation}
 for all $w\in W$ and $\sigma\geq0$. Since $\Pi(D(S))\subset D(A)$ then, using the polynomial stability of $T_{A}(t)$, it is easy to see that the first term in the right hand side of the equality (\ref{expressPiwith sigma}) converges to $0$ as $\sigma\rightarrow\infty$ for all $w\in D(S)$. Consequently, we have
 \begin{equation}
 \Pi w=\int_{0}^{\infty}T_{A}(t)\Delta T_{S}(-t)w dt.
 \end{equation}
 for all $w\in D(S)$. Conversely,  suppose that $\Delta\in \mathcal{L}(W,Z)$ is a conform operator for the semigroup $T_{A}(t)$. Then $$\Pi w:=\int_{0}^{\infty}T_{A}(t)\Delta T_{S}(-t)w dt$$ define a bounded linear operator from $D(S)$ to $Z$.
 By density argument, $\Pi$ can be extended to a linear bounded operator from $W$ to $Z$ noted also by $\Pi$.
 Let $w\in D(S)$ and $t>0$ be arbitrary. We have
 $$\begin{array}{ccl}
 T_{A}(t)\Pi w-\Pi w&=& \int_{0}^{\infty}T_{A}(t+\sigma)\Delta T_{S}(-\sigma)wd\sigma-\int_{0}^{\infty}T_{A}(\sigma)\Delta T_{S}(-\sigma)wd\sigma\\
&&\\
 &=& \int_{t}^{\infty}T_{A}(\sigma)\Delta T_{S}(t-\sigma)wd\sigma-\int_{0}^{\infty}T_{A}(\sigma)\Delta T_{S}(-\sigma)wd\sigma\\
 &&\\
 &=& \int_{0}^{\infty}T_{A}(\sigma)\Delta T_{S}(-\sigma)(T_{S}(t)w-w)d\sigma-\int_{0}^{t}T_{A}(\sigma)\Delta T_{S}(t-\sigma)wd\sigma.
 \end{array}$$
 Hence, for all $t\in(0,1)$, we have
 $$\begin{array}{ccl}
 \frac{T_{A}(t)\Pi w-\Pi w}{t}&=&\int_{0}^{\infty}T_{A}(\sigma)\Delta T_{S}(-\sigma)(\frac{T_{S}(t)w-w}{t})d\sigma-
 \frac{1}{t}\int_{0}^{t}T_{A}(\sigma)\Delta T_{S}(t-\sigma)wd\sigma\\
 &&\\
 &=&\Pi(\frac{T_{S}(t)w-w}{t})-\frac{1}{t}\int_{0}^{t}T_{A}(\sigma)\Delta T_{S}(t-\sigma)w d\sigma.
 \end{array}$$
 In the first term on the right hand side, we use the fact that $\frac{T_{S}(t)w-w}{t}\in D(S)$ for every $w\in D(S)$.
Since
 $\displaystyle\lim_{t\longrightarrow0^{+}}\frac{T_{S}(t)w-w}{t}=Sw$ and $\Pi\in\mathcal{L}(W,Z)$, then this first term converges to $\Pi Sw$ as $t\longrightarrow0^{+}$. From another hand, since
  the map \,$\sigma\longrightarrow T_{A}(\sigma)\Delta T_{S}(t-\sigma)w$ \,is continuous on $\mathbb{R}^{+}$ for every $t\in (0,1)$ and
 $T_{S}(t)$ is an isometric group on $W$ then the second term converges to $-\Delta w$. Therefore, $$\lim_{t\longrightarrow0^{+}}\frac{T_{A}(t)\Pi w-\Pi w}{t}\; \mbox{exists in Z}.$$
 Consequently,  $\Pi w\in D(A)$ and $A\Pi w=\Pi Sw-\Delta w$.
 \end{proof}
 \begin{rem}
 if $T_{A}(t)$ is exponentially stable, then by Corollary $8$ in \cite{Phong}, the operator equation $\Pi S=A\Pi+\Delta$ in $D(S)$ has a (unique) solution
 $\Pi\in \mathcal{L}(W,Z)$ for every $\Delta\in\mathcal{L}(W,Z)$ defined by the same formula
 $$\Pi w=\int_{0}^{\infty}T_{A}(t)\Delta T_{S}(-t)w dt $$ for all $w\in W$,
 that is  every operator $\Delta\in\mathcal{L}(W,Z)$ is conform for an exponentially stable $C_{0}$-semigroup.
 \end{rem}
 In the following proposition, we  give a sufficient condition of  the conformity of an  operator with a   polynomially stable $C_{0}$-semigroup.
 \begin{prop}\label{suffcondpoly}
 Let A generate a bounded polynomially stable $C_{0}$-semigroup $T_{A}(t)$ and let $\Delta\in\mathcal{L}(W,Z)$. If there exists $\varepsilon>0$ such that $\Delta\in\mathcal{L}(W,D((-A)^{\alpha+\varepsilon}))$, then the operator $\Delta$ is conform for the semigroup $T_{A}(t)$.
 \end{prop}
  \begin{proof}
 Since $T_{A}(t)$ is a bounded polynomially stable $C_{0}$-semigroup then, by inequality (\ref{polywithgamma}), with $\gamma=1+\frac{\varepsilon}{\alpha}$ and $\beta=1$, we have
 \begin{equation*}
  \| T_{A}(t)(-A)^{-\alpha-\varepsilon}\|\leq N_{\varepsilon}\,t^{-1-\frac{\varepsilon}{\alpha}},
 \end{equation*}
 or equivalently
 \begin{equation}
  \| T_{A}(t)z\|\leq N_{\varepsilon}\,t^{-1-\frac{\varepsilon}{\alpha}}\|z\|_{\alpha+\varepsilon}
 \end{equation}
 for all $z\in Z_{\alpha+\varepsilon}:=D((-A)^{\alpha+\varepsilon})$
 Let $a>0$ such that $\|\Delta w\|_{\alpha+\varepsilon}\leq a\,\|w\|$ for all $w\in D(S)$. Since $\Delta T_{S}(-t)w\in Z_{\alpha+\varepsilon}$, then
 $$\begin{array}{lll}\vspace{0.5cm}
   \|T_{A}(t)\Delta T_{S}(-t)w\| & \leq & N_{\varepsilon}\,t^{-1-\frac{\varepsilon}{\alpha}}\|\Delta T_{S}(-t)w\|_{\alpha+\varepsilon} \\\vspace{0.5cm}
    & \leq & a\,N_{\varepsilon}\,t^{-1-\frac{\varepsilon}{\alpha}}\|T_{S}(-t)w\| \\\vspace{0.5cm}
    & = & a\,N_{\varepsilon}\,t^{-1-\frac{\varepsilon}{\alpha}}\|w\|.
 \end{array}
$$ 
Thus, the function $t\longrightarrow T_{A}(t)\Delta T_{S}(-t)w$ is integrable on $[0,\infty).$ Consequently, 
  $$
\Pi w=\int_{0}^{\infty}T_{A}(t)\Delta T_{S}(-t)wdt
$$ define a bounded linear operator from $W$ to $Z$, and  the operator $\Delta$ is conform for the semigroup $T_{A}(t)$.
  \end{proof}
 One way to obtain the necessity of solvability of the regulator equations (\ref{reg eq}) for the solvability of the SFRP is to restrict the class of reference signals as follows (see \cite{Immonen}).
 \begin{defn}\cite{Immonen}
 The exogenous system generates admissible reference signals if for every $w\in W$ and each $Q\in\mathcal{L}(W,Y)$ we have:
  $QT_{S}(\cdot)w\in C_{0}^{+}(\mathbb{R},Y)$ only if $Qw=0$.
 \end{defn}
 This is the case if, for example, the operator $S$ in (\ref{exo}) generates a periodic $C_{0}$-group $T_{S}(t)$ on $W$. More generally, one has the following result.
 \begin{prop}\cite{Immonen}
 The exogenous system (\ref{exo}) generates admissible reference signals provided that at least one of the four conditions below holds:
 \begin{itemize}
 \item[$1$.] The reference signals $QT_{S}(\cdot)w$ are in $AP(\mathbb{R},Y)$ for all $Q\in \mathcal{L}(W,Y)$ and all $w\in W$.
 \item[$2$.] The spectrum $\sigma(S)$ is countable and $H$ does not contain a closed subspace which is isomorphic to $c_{0}$ (the space of
 sequences converging to $0$ with sup-norm).
 \item[$3$.] The spectrum $\sigma(S)$ is discrete.
 \item[$4$.] The space $W$ is of finite-dimensional.
 \end{itemize}
 \end{prop}
 The following theorem presents some conditions under which the solvability of the regulator equations (\ref{reg eq}) is
 necessary for the solvability of the SFRP.
 \begin{thm}
 Assume that the exosystem (\ref{exo}) generates admissible reference signals. If the SFRP is solvable for some control law $u(t)=Kz(t)+Lw(t)$ such that the operator $BL+P\in \mathcal{L}(W,Z)$ is conform for the semigroup $T_{A+BK}(t)$, then there exits $\Pi\in\mathcal{L}(W,Z)$ and
 $\Gamma\in\mathcal{L}(W,U)$ such that $\Pi(D(S))\subset D(A)$, $L=\Gamma-K\Pi$ and the regulator equations (\ref{reg eq}) are satisfied.
 \end{thm}
 \begin{proof}
 Since $BL+P$ is conform for $T_{A+BK}(t)$, by Lemma (\ref{suffcondpoly}), there exists $\Pi\in\mathcal{L}(W,Z)$  such that $\Pi(D(S))\subset D(A+BK)$ and
 \begin{equation*}
 \Pi S=(A+BK)\Pi+BL+P \quad {\mbox in}\;D(S).
 \end{equation*}
 Let $\Gamma=L+K\Pi\in \mathcal{L}(W,U)$. Then $\Pi$ and $\Gamma$ solve the first regulator equation of (\ref{reg eq}).
 Next, we show that also the second regulator equation is satisfied.
 By Lemma \ref{lemmaprelim}, we have
 $$\int_{0}^{t}T_{A+BK}(t-\sigma)(BL+P)T_{S}(\sigma)wd\sigma=\Pi T_{S}(t)w-T_{A+BK}(t)\Pi w$$ for all $w\in W$.
 Let $w_{0}\in W$ be arbitrary and take $z_{0}=\Pi w_{0}\in Z$. Then the corresponding tracking error $e(t)$, see (\ref{errorexpreswith(CPi-Q)}), is given by
 $$\begin{array}{ccl}
 e(t)&=&CT_{A+BK}(t)(z_{0}-\Pi w_{0})+(C\Pi-Q)T_{S}(t)w_{0}\\
 &=&(C\Pi-Q)T_{S}(t)w_{0}.
 \end{array}$$
 Now $(C\Pi-Q)T_{S}(\cdot)w_{0}\in C_{0}^{+}(\mathbb{R},Y)$ because the SFRP is solvable. Since the exosystem \eqref{exo} generates admissible reference signals and since $C\Pi-Q\in \mathcal{L}(W,Y)$, we must have $$C\Pi w_{0}-Qw_{0}=0.$$ \end{proof}
 By combining the above result with those in Section \ref{section suffcond}, we obtain the following complete characterization for the solvability of
  the regulator equations \eqref{reg eq} and the SFRP.  \begin{thm}\label{charac of sfrp}
  Let $T_{A}(t)$ be a bounded polynomially stable $C_{0}$-semigroup and assume that the exosystem (\ref{exo}) generates admissible reference signals. Then the SFRP is solvable using the control law $u(t)=Kz(t)+Lw(t)$, where $L\in\mathcal{L}(W,U)$ and $BL+P$ is conform for $T_{A+BK}(t)$, if and only if there exist $\Pi\in\mathcal{L}(W,Z)$ and $\Gamma\in \mathcal{L}(W,U)$ such that $\Pi(D(S))\subset D(A)$, $L=\Gamma -K\Pi$ and the regulator equations (\ref{reg eq}) are satisfied.\\
  \end{thm}
  Finally, if $A+BK$ generates a bounded polynomially stable $C_{0}$-semigroup and if we can solve the SFRP in such way that output regulation
   is polynomially fast, then we can dispense with the assumption that the exogenous system (\ref{exo}) only generates admissible reference signals.
   We arrive at another complete characterization for the solvability of SFRP.
\begin{thm}
  Assume that $A+BK$ generates a bounded polynomially stable $C_{0}$-semigroup $T_{A+BK}(t)$, then there exists $L\in\mathcal{L}(W,U)$ such that $BL+P$ is conform for the semigroup $T_{A+BK}(t)$ and such that the control law $u(t)=Kz(t)+Lw(t)$ solves the SFRP with
  \begin{equation}\label{ineqerror}
  \|e(t)\|\leq Mt^{-\beta}[\|z_{0}\|_{A}+\|w_{0}\|_{S}]
  \end{equation}
  for all $t>0$ and some $\beta,M>0$,   not depending on the initial conditions $z_{0}\in D(A)$ and $w_{0}\in D(S)$, if and only if $L=\Gamma -K\Pi$ where $\Pi\in\mathcal{L}(W,Z)$ and $\Gamma\in \mathcal{L}(W,U)$ satisfy the regulator equations (\ref{reg eq}).
  \end{thm}
  \begin{proof}
  From the above section, we need only to show the only if condition.
   Since $BL+P$ is conform for the semigroup $T_{A+BK}(t)$ then there exists $\Pi\in\mathcal{L}(W,Z)$
  such that $$\Pi S=(A+BK)\Pi+BL+P\hspace{0.5cm} in\; D(S).
$$
  Moreover, by (\ref{errorexpreswith(CPi-Q)}), the explicit tracking error is given by
$$e(t)=CT_{A+BK}(t)(z_{0}-\Pi w_{0})+(C\Pi-Q) T_{S}(t)w_{0}$$ for every $w_{0}\in W$ and $z_{0}\in Z$.
Assume that there exists $w_{0}\in D(S), w_{0}\neq0$ such that $\delta:=\|(C\Pi -Q)w_{0}\|>0$. Let $t_{0}>0$ be such that
\begin{equation}
M\,t_{0}^{-\beta}\,[\|\Pi\|_{\mathcal{L}(D(S),D(A+BK))}\|w_{0}\|_{S}+\|w_{0}\|_{S}]<\delta
\end{equation}
and set $\tilde{w}_{0}=T_{S}(-t_{0})w_{0}\in D(S)$ and $z_{0}=\Pi \tilde{w}_{0}\in D(A)$. Then the corresponding tracking error
satisfies
$$\begin{array}{ccl}
\|e(t_{0})\|&=&\|(C\Pi -Q)T_{S}(t_{0})\tilde{w}_{0}\|\\
&=&\|(C\Pi -Q)w_{0}\|=\delta\\
&>&M\,t_{0}^{-\beta}\,[\|\Pi\|_{\mathcal{L}(D(S),D(A+BK))}\|w_{0}\|_{S}+\|w_{0}\|_{S}].
\end{array}$$
Further,
 $$\|z_{0}\|_{A}=\|\Pi \tilde{w}_{0}\|_{A}\leq \|\Pi\|_{\mathcal{L}(D(S),D(A))}\|T_{S}(-t_{0})w_{0})\|_{S}
=\|\Pi\|_{\mathcal{L}(D(S),D(A))}\|w_{0}\|_{S}.$$ Therefore,
$$Mt_{0}^{-\beta}[\|z_{0}\|_{A}+\|w_{0}\|_{S}]\leq Mt_{0}^{-\beta}[\|\Pi\|_{\mathcal{L}(D(S),D(A))}\|w_{0}\|_{S}+\|w_{0}\|_{S}]<\|e(t_{0})\|$$
which is clearly a contradiction with the inequality (\ref{ineqerror}). Hence  $(C\Pi -Q)w_{0}=0$ for every $w_{0}\in D(S)$. Since $D(S)$ is dense in $W$ and $C\Pi -Q\in\mathcal{L}(W,Y)$, we conclude that
$$C\Pi -Q=0\quad \mbox{in} \;W.
$$
\end{proof}
\begin{rem}
In reviewing the previous proof, we note that the only if of the theorem remains true if we replace the inequality (\ref{ineqerror}) by the following
\begin{equation}
  \|e(t)\|\leq Mt^{-\beta}[\|z_{0}\|+\|w_{0}\|]
  \end{equation}
  for all $t>0$, some $\beta,M>0$ which do not depend on the initial conditions $z_{0}\in Z$ and $w_{0}\in W$.
\end{rem}

\section{Application: Periodic tracking for polynomially stabilizable SISO systems}
In this section, we shall solve the SFRP and the regulator equations (\ref{reg eq}) in the case that the reference and disturbance signals are periodic functions. Under some assumptions, in order to solve  SFRP, it is sufficient to verify that the operator $BL+P$ is conform for the semigroup $T_{A+BK}(t)$ for some operator $L\in\mathcal{L}(W,Z)$. Hence, the theorem (\ref{charac of sfrp}) allows to solve SFRP.  \\
We consider the system (\ref{plant}) with $U=Y=\mathbb{C}$. Assume that $Z$ is a Hilbert space with an inner product $\left<\cdot,\cdot\right>$. Define the operator $A\;:\;D(A)\subset Z\longrightarrow Z$  by
$$Az=\sum_{n\in\mathbb{Z}}\mu_{n}
\left<z,\psi_{n}\right>\phi_{n}$$
$$D(A)=\{z\in\mathbb{Z}\;|\;\sum_{n\in\mathbb{Z}}|\mu_{n}|^{2}|\left<z,\psi_{n}\right>|^{2}<\infty\}$$
where $(\phi_{n})_{n\in\mathbb{Z}}$ is a Riesz basis of $Z$ and $(\psi_{n})_{n\in\mathbb{Z}}$ is the corresponding biorthogonal sequence.
We assume that the set $\{\mu_{n}\}_{n\in\mathbb{Z}}$ is contained in $\mathbb{C}^{-}$ and that it has no accumulation points on $i\mathbb{R}$.  Assume further that the eigenvalues $\{\mu_{n}\}$ of the operator $A$ satisfies the following property:
there exist constants $\alpha,\;c>0$ and $d>0$ such that
\begin{equation}\label{geometric assumption}
Re\,\mu_{n}\leq\;-\,\frac{c}{|Im\,\mu_{n}|^{\alpha}}\qquad \mbox{if}\;|Im\,\mu_{n}|\geq\,d.
\end{equation}
It is clear that $A$ generates a bounded   $C_{0}$-semigroup, and  since $A$ is similar to a normal operator then, due to \cite[Proposition 4.1]{Batkai}, the above assumptions also imply it generates  polynomially stable $C_{0}$-semigroup and  $$
\|T_{A}(t)A^{-1}\|\leq N\,t^{-1/\alpha}
$$ for all $t>0$, or  equivalently
$$\|T_{A}(t)(-A)^{-\alpha}\|\leq \frac{N}{t}$$ for all $t>0$.

Now we introduce an interesting functions space whose elements can be generated by the exosystem ($2.2$) if its parameters are suitably chosen. Let $p>0$, let $\omega_{k}=\frac{2\pi k}{p}$ for $k\in\mathbb{Z}$, let $(f_{k})_{k\in\mathbb{Z}}\subset\mathbb{R}$ such that $f_{k}\geq1,\;k\in\mathbb{Z}$ and $(f_{k}^{-1})_{k\in\mathbb{Z}}\in \ell^{2}$\;(the space of square summable complex sequences). We define the state space of the exosystem as follows
\begin{equation*}W=\{y:\mathbb{R}\longrightarrow\mathbb{C}\,|\,y(t)=\sum_{k\in\mathbb{Z}}y_{k}e^{i\omega_{k}t},\,\sum_{k\in\mathbb{Z}}|y_{k}|^{2}f_{k}^{2}<\infty\;,(y_{k})_{k\in I}\subset\mathbb{C}\}.
\end{equation*}
Put $\theta_{k}(t)=e^{i\omega_{k}t}$ for all $t\in\mathbb{R}$ and  $k\in\mathbb{Z}$. It is clear that the set $\{\theta_{k}\}_{k\in\mathbb{Z}}$ form an orthonormal basis in $W$ with the $L^{2}$-inner product which is denoted by $\left<\cdot,\cdot\right>_{L^{2}}$. But it shall be interesting to use the inner product in $W$ defined by $$\left<u,y\right>_{f}=\sum_{k\in\mathbb{Z}}u_{k}\overline{y_{k}}f_{k}$$ where $u_{k}=\left<u,\theta_{k}\right>_{L^{2}}$ and $y_{k}=\left<y,\theta_{k}\right>_{L^{2}}$ for all $k\in\mathbb{Z}$. It is easy to see that $W$ is a Hilbert space with the inner product $\left<\cdot,\cdot\right>_{f}$ and   the corresponding norm is $$\|y\|_{f}=\sqrt{\sum_{k\in\mathbb{Z}}|\left<y,\theta_{k}\right>|^{2}f_{k}^{2}}.
$$ 
Thus the set $\{\theta_{k}\}_{k\in\mathbb{Z}}$, with this inner product, form an orthogonal basis in $W$, with $\|\theta_{k}\|_{f}=f_{k}$ for all $k\in\mathbb{Z}$.
Define the operator
\begin{equation*}
S=\sum_{k\in\mathbb{Z}}i\omega_{k}<\cdot,\theta_{k}>_{L2}\theta_{k},\quad 
D(S)=\{y\in W\;|\;\sum_{k\in\mathbb{Z}}\omega_{k}^{2}|\left<y,\theta_{k}\right>_{L^{2}}|^{2}f_{k}^{2}<\infty\}.
\end{equation*}
Notice that the operator $S$ generates an isometric $C_{0}$-group on $W$ given by
\begin{equation*}
T_{S}(t)=\sum_{k\in\mathbb{Z}}
e^{i\omega_{k}t}\left<\cdot,\theta_{k}\right>\theta_{k}.
\end{equation*}
We fix $P\in\mathcal{L}(W,Z)$ and $Q=\delta_{0}$ (Dirac mass at $0$). Then the exosystem (\ref{exo}) with $W,S,Q,P\in\mathcal{L}(W,Z)$ and $w_{0}\in W$ can generate all reference signals in $W$ and only those. Moreover, every reference function $y_{r}\in W$ is generated by the choice $w_{0}=y_{r}\in W$. In fact $$\delta_{0}T_{S}(t)h=h(x+t)|_{x=0}=h(t)$$ for every $h\in W$ and $t\in\mathbb{R}$,  see  \cite{Immonen} for more details. Notice that $Q=\delta_{0}\in\mathcal{L}(W,Z)$ because if $y=\sum_{k\in\mathbb{Z}}\left<y,\theta_{k}\right>_{L^{2}}\theta_{k}$, we have
\begin{align*}
  |\delta_{0}y| & =  |\sum_{k\in\mathbb{Z}}\left<y,\theta_{k}\right>_{L^{2}}| \\
   & \leq  \sqrt{\sum_{k\in\mathbb{Z}}f_{k}^{-2}}\sqrt{\sum_{k\in\mathbb{Z}}|\left<y,\theta_{k}\right>_{L^{2}}|^{2}f_{k}^{2}} \\
   & =  c\|y\|_{f}.
\end{align*}
Our goal is to construct a control law $u(t)=Kz(t)+Lw(t)$ for the asymptotic tracking of periodic reference signals in the presence of the disturbance $\mathcal{U}_{d}(t)$. Since $A$ generates a bounded polynomially stable $C_{0}$-semigroup, we can choose $K=0$. To this end we need some assumptions\\
\textbf{Assumption $1$:} $H(i\omega_{k}):=CR(i\omega_{k},A)B\neq0$ for all $k\in\mathbb{Z}$.\\
\textbf{Assumption $2$:} $(H(i\omega_{k})^{-1}[1-H_{d}(k)]f_{k}^{-1})_{k\in\mathbb{Z}}\in\ell^{2}$, where $H_{d}(k):=CR(i\omega_{k},A)P\theta_{k}.$\\
With these assumptions, it is clear that the operator
\begin{equation}
Ly:=\sum_{k\in\mathbb{Z}}H(i\omega_{k})^{-1}[1-H_{d}(k)]\left<y,\theta_{k}\right>_{L^{2}}
\end{equation}
is bounded from $W$ to $\mathbb{C}$.

The main result of this section is the following.
\begin{thm}\label{Applic per}
Assume that \textbf{Assumption $1$} and \textbf{Assumption $2$} are satisfied. Suppose further that $BL+P$ is a conform operator for the semigroup $T_{A}(t)$.
 Then the SFRP is solvable using $u(t)=Lw(t)$. Moreover, for every $y_{r}\in W$ the corresponding control law $u_{y_{r}}$ which achieves the asymptotic tracking of $y_{r}$ in the presence of the disturbance $\mathcal{U}_{d}(t)$ is given, for all $t\geq0$, by
$$u_{y_{r}}(t)=\sum_{k\in\mathbb{Z}}H(i\omega_{k})^{-1}[1-H_{d}(k)]\left<y_{r},\theta_{k}\right>_{L^{2}}e^{i\omega_{k}t}.$$
\end{thm}
\proof
Since $BL+P$ is conform for the semigroup $T_{A+BK}(t)$ then there exists $\Pi\in\mathcal{L}(W,Z)$
such that $$\Pi S=(A+BK)\Pi+BL+P.$$ Furthermore, $$\Pi y=\int_{0}^{\infty}T_{A+BK}(t)(BL+P)T_{S}(-t)ydt$$ for all $y\in D(S)$. In particular we have,
\begin{equation*}
\Pi \theta_{k}=\int_{0}^{\infty}T_{A}(t)(BL+P)T_{S}(-t)\theta_{k}dt=\int_{0}^{\infty}e^{-i\omega_{k}t}T_{A}(t)(BL+P)\theta_{k}dt
\end{equation*}
 for all $k\in\mathbb{Z}$.
 Hence,
 \begin{equation}
 \Pi\theta_{k}=R(i\omega_{k},A)(BL+P)\theta_{k}
 \end{equation}
 and therefore,
\begin{equation}\label{(5.4)}
C\Pi\theta_{k}=CR(i\omega_{k},A)(BL+P)\theta_{k}.
\end{equation}
On the other hand, by  \textbf{Assumption $1$}  and \textbf{Assumption $2$}, we have $$H(i\omega_{k})=CR(i\omega_{k},A)B\quad\mbox{and}\quad H_{d}(k)=CR(i\omega_{k},A)P\theta_{k}.$$
Then,
\begin{align*}
  CR(i\omega_{k},A)(BL+P)\theta_{k} & =  
H(i\omega_{k})L\theta_{k}+H_{d}(k) \\
   & =  H(i\omega_{k}) H(i\omega_{k})^{-1}[1-H_{d}(k)]+H_{d}(k) \\
   & =  1=\delta_{0}\theta_{k}.
\end{align*}
Therefore, from (\ref{(5.4)}), we deduce that $$C\Pi\theta_{k}=\delta_{0}\theta_{k}.$$
 Now, consider $y=\sum_{k\in\mathbb{Z}}\left<y,\theta_{k}\right>\theta_{k}\in W$. Since $\Pi\in\mathcal{L}(W,Z)$ then
 \begin{equation*}
 C\Pi y=\sum_{k\in\mathbb{Z}}\left<y,\theta_{k}\right>C\Pi\theta_{k}=\sum_{k\in\mathbb{Z}}\left<y,\theta_{k}\right>=y(0)=\delta_{0} y.
 \end{equation*}
 Consequently, the second regulator equation is also satisfied and  $SFRP$ is solvable using $u(t)=Lw(t)$.
More precisely, for $y_{r}\in W$, the corresponding control law $u_{y_{r}}$ which achieves the asymptotic tracking of $y_{r}$ is given, for all $t\geq0$, by
\begin{align*}
  u_{y_{r}}(t) & = LT_{S}(t)y_{r} \\
   & = \sum_{k\in\mathbb{Z}}\left<y_{r},\theta_{k}\right>_{L^{2}}
e^{i\omega_{k}t}L\theta_{k} \\
   & = \sum_{k\in\mathbb{Z}}H(i\omega_{k})^{-1}[1-H_{d}(k)]\left<y_{r},\theta_{k}\right>_{L^{2}}e^{i\omega_{k}t}.
\end{align*}
\begin{rem}
$\bullet$ Under the assumptions of the above theorem, the operator $\Pi$ can be given explicitly by
$$\Pi y=\sum_{k\in\mathbb{Z}}\left<y,\theta_{k}\right>R(i\omega_{k},A)(BL+P)\theta_{k}.$$
$\bullet$ If there exists $\varepsilon>0$ such that $(BL+P)\in\mathcal{L}(W,Z_{\alpha+\varepsilon})$ then the operator $BL+P$ is conform
for the semigroup $T_{A}(t)$ and therefore we can apply  results of the above theorem.
\end{rem}
Next, we present two concrete examples of operators $A$ considered above.
\begin{exm}[\textbf{Periodic tracking of perturbed wave equation}]
\end{exm}
Consider the following perturbed one-dimensional wave equation on $(0.1)$\\
\begin{align*}
\frac{\partial^{2}v}{\partial t^{2}}(x,t) & =  \frac{\partial^{2}v}{\partial x^{2}}(x,t)+b_{0}(x)[h_{1}(x)v(x,t)+h_{2}(x)\frac{\partial v}{\partial t}(x,t)] \\
v(0,t) &=  v(1,t)=\;0 \\
v(x,0)&=v_{0}, \quad \frac{\partial v}{\partial t}(x,0)=v_{1},
\end{align*}
where $b_{0}(x)=\sqrt{3}(1-x)$, $h_{1}$ and $h_{2}$ are defined in \cite[Theorem 13]{Paunonen} by
$$
h=\left[\begin{array}{c}
                  h_{1} \\
                  h_{2}
                \end{array}\right]=-\frac{\nu\pi^{2}}{\sqrt{3}}\sum_{k\neq0}\frac{\overline{\alpha_{k}}}{k}\psi_{k},\quad \mbox{where}\;\alpha_{k}=\prod_{l\neq0,k}(1+i\frac{\nu}{l^{2}(l-k)})$$
                for some $0<\nu\leq1$ with $\psi_{k}(x)=\frac{1}{\lambda_{k}}\left[\begin{array}{c}
                                    sin(k\pi x) \\
                                    \lambda_{k}sin(k\pi x)
                                  \end{array}\right]$ and $\lambda_{k}=ik\pi$ for $k\in \mathbb{Z}\backslash\{0\}$.\\
Define the operator $A_{0}\,:\,D(A_{0})\subset L^{2}(0,1)\longrightarrow L^{2}(0,1)$ by
$$A_{0}=-\frac{d^{2}}{dx^{2}}$$ with domain
$D(A_{0})=\left\{v\in L^{2}(0,1)|v,v' \mbox{ abs. cont. } v''\in L^{2}(0,1), v(0)=v(1)=0\right\}.$
The operator $A_{0}$ has a positive self-adjoint square root $A_{0}^{1/2}$ and the space $Z=D(A_{0}^{1/2})\times L^{2}(0,1)$ equipped with an inner product $$\langle v,w\rangle_{Z}=\langle A_{0}^{1/2}v_{1},A_{0}^{1/2}w_{1}\rangle_{L^{2}}+\langle v_{2},w_{2}\rangle_{L^{2}}$$ is a Hilbert space. Next we introduce $$z=\left[\begin{smallmatrix}
                  v \\
                  \frac{dv}{dt}
                \end{smallmatrix}\right], \quad A=\left[\begin{matrix}
                                    0 & I \\
                                        -A_{0} & 0                                          
\end{matrix}\right], \quad \Delta=\left[\begin{matrix}
                0 & 0 \\
b_{0}h_{1} & b_{0}h_{2}                                          
\end{matrix}\right],
$$
$$
D(A)=D(A_{0})\times D(A_{0}^{1/2}).$$
Notice that $\Delta\in\mathcal{L}(Z)$. Introduce also the control operator $B\xi=b\xi=\left[\begin{smallmatrix}
                  0 \\
                  b_{1}
                \end{smallmatrix}\right]\xi$,    where $b_{1}(x)=x(1-x)$ and the observation operator $C=\sum_{k\neq0}\left<\cdot,\phi_{k}\right>\overline{\left<b,\phi_{k}\right>}\in \mathcal{L}(Z,\mathbb{C})$. Thus the above perturbed wave equation can be written as
\begin{align}
 \begin{cases}
    \dot{z}(t)  &=  (A+\Delta)z(t)+Bu(t),
\\
  z(0)&= z_{0}=
\left[\begin{smallmatrix}
                  v_{0} \\
                  v_{1}
                \end{smallmatrix}
\right]\in Z,\; \\
    y(t) & =  Cz(t),\;t\geq0. 
  \end{cases}
\end{align}
The eigenvalues of $A$ are $\lambda_{k}$ and the corresponding eigenvectors are  $\psi_{k}$ form an orthonormal basis in $Z$.
Due to Paunonen \cite{Paunonen}, the perturbed operator $\tilde{A}:=A+\Delta$ is a Riesz-spectral operator and that $\sigma(\tilde{A})=\{\frac{-\nu\pi}{k^{2}}+ik\pi\}_{k\neq0}$. Further, the eigenvectors $\phi_{k}$ of $\tilde{A}$ form a Riesz basis of $Z$. Since $\tilde{A}$ is similar to a normal operator then, due to \cite[Theorem 4.1]{Batkai}, $\tilde{A}$ generates a polynomially stable $C_{0}$-semigroup $T_{\tilde{A}}(t)$ and we have
\begin{equation}\label{polywave}
\|T_{\tilde{A}}(t)\tilde{A}^{-1}\|\leq \frac{N}{\sqrt{t}}
\end{equation}
for all $t>0$ and some $N>0$.
Since $T_{\tilde{A}}(t)$ is bounded then the estimate (\ref{polywave}) is equivalent to
\begin{equation}\label{polywith r}
\|T_{\tilde{A}}(t)(-\tilde{A})^{-2r}\|\leq \frac{N}{t^{r}}
\end{equation}
for all $r>0,t>0$. Since $\sigma(\tilde{A})\subset\mathbb{C}^{-}$ and it has no finite accumulation points on $i\mathbb{R}$, the operator $-\tilde{A}$ is an invertible sectorial operator. Using the fractional domains of the operator $-\tilde{A}$, we have, for $\beta\geq0$, 
$$
Z_{\beta}=D((-\tilde{A})^{\beta})=\{z\in Z|\sum_{k\neq0}|\mu_{k}|^{2\beta}|<z,\phi_{k}>|^{2}<\infty\}.
$$
The space $(Z_{\beta},\|\cdot\|_{\beta})$  is a Hilbert space with norm defined by
$$\|z\|^{2}_{\beta}=\sum_{k\neq0}|\mu_{k}|^{2\beta}|<z,\phi_{k}>|^{2}$$
where $\mu_{k}=\frac{-\nu\pi}{k^{2}}+ik\pi$ for $k\neq0$.\\
In order to apply the above results, we shall consider asymptotic tracking of the $p$-periodic reference signals in the above space $W$. Consider the exogenous system (\ref{exo}) with $W,S,Q=\delta_{0},P=0$ and $w_{0}\in W$.
To verify the {\bf Assumption $1$}, we have
\begin{equation}
R(i\omega_{k},\tilde{A})=\sum_{n\neq0}\frac{<\cdot,\phi_{n}>\Phi_{n}}{i\omega_{k}-\mu_{n}}
\end{equation}
for all $k\in\mathbb{Z}$. Hence the transfer function evaluated at the frequencies $i\omega_{k}$ is
\begin{equation}
H(i\omega_{k})=CR(i\omega_{k},\tilde{A})B=\sum_{n\neq0}\frac{|<b,\phi_{n}>|^{2}}{i\omega_{k}-\mu_{n}}.
\end{equation}
For a suitable choice of $b\in Z$, we can verify that $H(i\omega_{k})\neq0$ for all $k\in\mathbb{Z}$.
 On the other hand, put
\begin{equation}\label{express L}
Ly:=\sum_{k\in\mathbb{Z}}H(i\omega_{k})^{-1}\left<y,\theta_{k}\right>_{L^{2}},
\end{equation}
for $y\in W$ such that $L\in\mathcal{L}(W,\mathbb{C})$. We saw that the operator $L$ is bounded if and only if the condition
\begin{equation}\label{condition of L}
\sum_{k\in\mathbb{Z}}|H(i\omega_{k}|^{-2}|f_{k}|^{-2}<\infty.
\end{equation}
holds. Now, we show that $BL$ is a conform operator. To this end we verify that $b\in Z_{2+\varepsilon}=D((-\tilde{A})^{2+\varepsilon})$ for some $\varepsilon>0$.  We have
$$\left<b,\psi_{k}\right>_{Z}=\left<b_{1},\sin k\pi\cdot\right>_{L^{2}}=\int_{0}^{1}b_{1}(x)\sin k\pi x\,dx$$ for all $k\in \mathbb{Z}$. An easy computation shows that
 $$<b,\psi_{k}>_{Z}=\frac{2(1-(-1)^{k})}{k^{3}\pi^{3}}.$$
Since $\mid\mu_{k}\mid=O(k^{2})$ as $k\longrightarrow\infty$ then $$\mid\mu_{k}\mid^{2(2+\varepsilon)}\mid<b,\psi_{k}>_{Z}\mid^{2}=O(k^{2\varepsilon-2})$$
Hence the series $$\sum_{k\neq0}|\mu_{k}|^{2(2+\varepsilon)}|<z,\psi_{k}>|^{2}<\infty$$ for $0<\varepsilon<\frac{1}{2}$. Thus if we choose  $0<\varepsilon<\frac{1}{2}$, the operator $BL$ is bounded from $W$ to $Z_{2+\varepsilon}$ and due to the Proposition (\ref{suffcondpoly}), we conclude that $BL$ is a conform operator for the semigroup $T_{\tilde{A}}(t)$ and the first regulator equation of (\ref{reg eq}) has a bounded solution given by
$$\Pi y:=\int_{0}^{\infty}T_{\tilde{A}}(t)BL T_{S}(-t)y dt$$
for all $y\in D(S)$. Thus, we can use  Theorem \ref{Applic per} to conclude that  $SFRP$ is solvable using $u(t)=Lw(t)$.
More precisely, for $y_{r}\in W$, the corresponding control law $u_{y_{r}}$ which achieves the asymptotic tracking of $y_{r}$ is given, for all $t\geq0$, by
   $$u_{y_{r}}(t)=\sum_{k\in\mathbb{Z}}H(i\omega_{k})^{-1}\left<y_{r},\theta_{k}\right>_{L^{2}}e^{i\omega_{k}t}.$$

\begin{exm}
\end{exm}
Let $Z$ be a Hilbert space with an inner product $<\cdot,\cdot>$ and an orthonormal basis $(\psi_{n})_{n\in\mathbb{Z}}$. Consider a linear control system (\ref{plant}) where $A=\sum_{n\in\mathbb{Z}}\mu_{n}
\left<\cdot,\psi_{n}\right>\psi_{n}$ with $\mu_{n}=-\frac{1}{1+|n|}+i\omega_{n}$ and
$$D(A)=\{z\in Z\,|\,\sum_{n\in\mathbb{Z}}|\mu_{n}|^{2}|\left<z,\psi_{n}\right>|^{2}<\infty\},$$ $B\xi=\psi_{0}\xi$ for all $\xi\in\mathbb{C}$,\,$C=\left<\cdot,\psi_{0}\right>$ and $P=0$. As in the beginning of this section, we see that $A$ generates a bounded polynomially stable $C_{0}$-semigroup $T_{A}(t)$ with $\alpha=1$.\\
we shall consider asymptotic tracking of  $p$-periodic reference signals in the space $\mathcal{H}_{\gamma}=W$ with $f_{n}=\sqrt{1+\omega_{n}^{2}}^{\gamma}$ where $\gamma>\frac{1}{2}$. Notice that the condition $\gamma>\frac{1}{2}$ is necessary so that $(f_{n}^{-1})_{n\in\mathbb{Z}}\in\ell^{2}$ but not sufficient for the boundedness of the operator $L$ defined as above.
Since $A$ is a Riesz spectral operator \cite{Curtain}, we have $$R(\lambda,A)=\sum_{n\in\mathbb{Z}}\frac{1}{\lambda-\mu_{n}}\left<\cdot,\psi_{n}\right>\psi_{n},\quad \lambda\in \rho(A).
$$
Hence we have $R(\lambda,A)B=\frac{1}{\lambda+1}\psi_{0}$. Since $(\psi_{n})_{n\in\mathbb{Z}}$ is an orthonormal basis then the transfer function of  the plant satisfies $$H(i\omega_{k})=CR(i\omega_{k})B=\frac{1}{1+i\omega_{k}}\neq0,\quad  k\in\mathbb{Z}.
$$
Now,  we define the operator $L$ as in (\ref{express L}) and show that $L$ is bounded for a suitable choice of $\gamma$. In fact, by using the Schwartz inequality we have $$\|Ly\|\leq\sum |[1+i\omega_{k}]\left<y,\theta_{k}\right>_{L^{2}}|\leq \sqrt{\sum_{n\in\mathbb{Z}}|\left<y,\theta_{k}\right>_{L^{2}}|^{2}(1+i\omega_{n}^{2})^{\gamma}}
\sqrt{\sum_{n\in\mathbb{Z}}\frac{1+\omega_{n}^{2}}{(1+\omega_{n}^{2})
^{\gamma}}}.
$$
Since $\|y\|_{f}=\sqrt{\sum_{n\in\mathbb{Z}}|\left<y,\theta_{k}\right>_{L^{2}}|^{2}(1+i\omega_{n}^{2})^{\gamma}}$ then the operator $L$ is bounded if $$\sum_{n\in\mathbb{Z}}\frac{1+\omega_{n}^{2}}{(1+\omega_{n}^{2})^{\gamma}}
<\infty.$$ This can be verified whenever $\gamma>\frac{3}{2}$. Now, to apply the Theorem \ref{Applic per} we shall verify that $BL$ is conform operator for the semigroup $T_{A}(t)$. We have $B\xi=\psi_{0}\xi$ for all $\xi\in\mathbb{C}$. Since $\psi_{0}\in Z_{\eta}=D((-A)^{\eta})$ for all $\eta>0$, clearly $B\in\mathcal{L}(\mathcal{H}_{\gamma},\mathbb{C})$ and consequently $BL\in\mathcal{L}(\mathcal{H}_{\gamma},Z_{\eta})$ whenever $\gamma>^{}\frac{3}{2}$. 
Due to  Proposition (\ref{suffcondpoly}), we conclude that $BL$ is a conform operator for the semigroup $T_{\tilde{A}}(t)$ and the results of Theorem (\ref{Applic per}) can be applied to conclude that, for $\gamma>\frac{3}{2}$ the control law $u(t)=LT_{S}(t)y_{r}=\displaystyle
\sum_{n\in\mathbb{Z}}\left<y_{r},\theta_{n}\right>_{L^{2}}[1+i\omega_{n}]e^{i\omega_{n}t}$ achieves asymptotic tracking of an arbitrary $y_{r}\in\mathcal{H}_{\gamma}$.

\end{document}